\documentclass[12pt]{amsart}

\usepackage{algorithmic}
\usepackage[section]{algorithm}

\usepackage{fullpage}
\usepackage{amsfonts,amssymb,amscd,amsmath,enumerate,verbatim,color,xcolor}
\usepackage[latin1]{inputenc}
\usepackage{latexsym}
\usepackage{tikz}
\usepackage{graphicx}
\usepackage{here}
\usepackage{mathptmx}
\usepackage{hyperref}
\hypersetup{colorlinks,linkcolor=blue,citecolor=blue,urlcolor=blue}

\def\NZQ{\mathbb}

\def\ZZ{{\NZQ Z}}
\def\RR{{\NZQ R}}

\def\ab{{\mathbf a}}

\def\cb{{\mathbf c}}

\def\gb{{\mathbf g}}
\def\hb{{\mathbf h}}

\def\vb{{\mathbf v}}
\def\ub{{\mathbf u}}
\def\wb{{\mathbf w}}
\def\xb{{\mathbf x}}
\def\yb{{\mathbf y}}
\def\zb{{\mathbf z}}

\def\opn#1#2{\def#1{\operatorname{#2}}}
\opn\conv{conv}
\opn\Vol{Vol}
\opn\ord{ord}
\opn\supp{supp}
\opn\im{Im}
\opn\Ker{Ker}
\def\red{{\rm red}}

\def\Lc{{\mathcal L}}

\def\Pc{{\mathcal P}}
\def\Qc{{\mathcal Q}}

\def\hei{{\rm ht}}

\newtheorem{Theorem}{Theorem}[section]
\newtheorem{Lemma}[Theorem]{Lemma}
\newtheorem{Corollary}[Theorem]{Corollary}
\newtheorem{Proposition}[Theorem]{Proposition}

\theoremstyle{definition}

\newtheorem{Example}[Theorem]{Example}
\newtheorem{Definition}[Theorem]{Definition}

\numberwithin{equation}{section}

\title{Classification and counting of Gorenstein simplices with $h^*$-polynomial $1+t^k+\cdots+t^{(v-1)k}$}
\author{Akiyoshi Tsuchiya}

\address{Akiyoshi Tsuchiya,
Department of Information Science,
Faculty of Science,
Toho University,
2-2-1 Miyama, Funabashi, Chiba 274-8510, Japan}
\email{akiyoshi@is.sci.toho-u.ac.jp}

\begin{document}
\maketitle

\begin{abstract}
Hibi, Yoshida, and the author classified Gorenstein simplices which are not
lattice pyramids and whose \(h^*\)-polynomials are of the form
\(1+t^k+t^{2k}+\cdots+t^{(v-1)k}\)
when \(v\) is a prime number or the product of two prime numbers. They also
conjectured that, for general \(v\), the number of unimodular equivalence
classes of such simplices depends only on the divisor lattice of \(v\). This
paper proves the conjecture by giving a constructive classification of
Gorenstein simplices whose \(h^*\)-polynomials are of this form. More precisely, their
unimodular equivalence classes are shown to be parametrized by strict divisor
chains in the divisor lattice of \(v\) together with certain recursive
combinatorial data. As a consequence, an explicit formula for the number of
equivalence classes in terms of the divisor lattice of \(v\) is obtained.
\end{abstract}

\section{Introduction}

A central problem in Ehrhart theory is to understand how much information about a lattice polytope is encoded in its \(h^*\)-polynomial. In general, many non-equivalent lattice polytopes may have the same \(h^*\)-polynomial. It is therefore natural to ask for classification results within special classes of lattice polytopes. In this paper, we study such a classification problem for Gorenstein simplices whose \(h^*\)-polynomials have the special form
\[
1+t^k+t^{2k}+\cdots+t^{(v-1)k}.
\]

We recall the basic terminology. A \emph{lattice polytope} is a convex polytope all of whose vertices have integer coordinates. Two lattice polytopes \(\Pc,\Qc\subset \RR^d\) are said to be \emph{unimodularly equivalent} if there exist a unimodular matrix \(U\in\ZZ^{d\times d}\) and a lattice point \(\wb\in\ZZ^d\) such that
\[
\Qc=f_U(\Pc)+\wb,
\]
where \(f_U:\RR^d\to\RR^d\) is the linear map defined by \(f_U(\xb)=\xb U\). In classification problems for lattice polytopes, one usually works up to unimodular equivalence.

Let \(\Pc\subset \RR^d\) be a lattice polytope of dimension \(d\). For a positive integer \(m\), set
$L_{\Pc}(m)=|m\Pc\cap \ZZ^d|$,
where $m\Pc:=\{m\xb:\xb\in\Pc\}$.
Ehrhart proved in \cite{Ehrhart} that \(L_{\Pc}(m)\) is a polynomial in \(m\) of degree \(d\) with constant term \(1\). The generating series
${\rm Ehr}_{\Pc}(t)=1+\sum_{m=1}^\infty L_{\Pc}(m)t^m$
can be written as
\[
{\rm Ehr}_{\Pc}(t)=\frac{h^*(\Pc,t)}{(1-t)^{d+1}},
\]
where \(h^*(\Pc,t)\) is a polynomial of degree at most \(d\) with nonnegative integer coefficients \cite{Stanley_nonnegative}. This polynomial is called the \(h^*\)-polynomial of \(\Pc\). If
$h^*(\Pc,t)=\sum_{i=0}^d h_i^*t^i$,
then \(h_0^*=1\), \(h_1^*=|\Pc\cap \ZZ^d|-(d+1)\), and \(h_d^*=|{\rm int}(\Pc)\cap \ZZ^d|\). Moreover,
$h^*(\Pc,1)=\sum_{i=0}^d h_i^*$ coincides with
the normalized volume $\Vol(\Pc)$ of \(\Pc\). We refer the reader to \cite{BeckRobins2015book} for background on Ehrhart theory.

The \emph{lattice pyramid} over \(\Pc\) is
\[
{\rm Pyr}(\Pc):=\conv(\Pc\times\{0\},(0,\ldots,0,1))\subset \RR^{d+1}.
\]
More generally, a lattice polytope is called a lattice pyramid if it is obtained by successively taking lattice pyramids over a lower-dimensional lattice polytope. Since taking lattice pyramids preserves the \(h^*\)-polynomial, and since for a fixed \(h^*\)-polynomial there are only finitely many unimodular equivalence classes up to lattice pyramids, it is natural to exclude lattice pyramids in classification problems with fixed \(h^*\)-polynomials.

A lattice polytope \(\Pc\subset \RR^d\) is called \emph{reflexive} if the origin belongs to the interior of \(\Pc\) and the dual polytope of \(\Pc\) is again a lattice polytope. A lattice polytope \(\Pc\) is called \emph{Gorenstein of index \(r\)} if \(r\Pc\) is unimodularly equivalent to a reflexive polytope. Gorenstein polytopes are characterized by palindromic \(h^*\)-polynomials. In fact, \(\Pc\) is Gorenstein if and only if
\[
h^*(\Pc,t)=1+h_1^*t+\cdots+h_s^*t^s \qquad (h_s^*\neq 0)
\]
satisfies \(h_i^*=h_{s-i}^*\) for all \(0\le i\le s\) \cite{DeNegriHibi}. Thus the polynomial \(1+t^k+t^{2k}+\cdots+t^{(v-1)k}\) naturally appears in the Gorenstein setting.
Gorenstein polytopes arise naturally in several areas such as toric geometry, commutative algebra, and mirror symmetry, and in each fixed dimension there exist only finitely many Gorenstein polytopes up to unimodular equivalence.
For these reasons, the classification of Gorenstein polytopes is a natural and important problem.

In this paper, we focus on Gorenstein simplices. A key tool is the finite subgroup model for lattice simplices. To a lattice simplex \(\Delta\), one can associate a finite subgroup \(\Lambda_\Delta\), and the \(h^*\)-polynomial of \(\Delta\) is recovered as the height enumerator of \(\Lambda_\Delta\). This model also translates lattice pyramids and unimodular equivalence into simple conditions on the associated finite groups, and it has proved useful in the classification of lattice simplices \cite{BatyrevHof, HigashitaniNillTsuchiya,TsuchiyaGorBinary}. We recall the precise correspondence in Section~2.

In \cite{TsuchiyaGorsimp}, the author classified Gorenstein simplices which are not lattice pyramids and whose normalized volumes are prime numbers.
Moreover, for such a simplex, the \(h^*\)-polynomial is necessarily of the form
\[
1+t^k+t^{2k}+\cdots+t^{(p-1)k}
\]
for some prime number \(p\) and some positive integer \(k\).
Motivated by this result, Hibi, Yoshida, and the author studied in \cite{HibiTsuchiyaYoshida} Gorenstein simplices which are not lattice pyramids and whose \(h^*\)-polynomials are of the form
\[
1+t^k+t^{2k}+\cdots+t^{(v-1)k}
\]
with fixed positive integers \(k\) and \(v\ge2\). They classified such simplices when \(v\) is the product of two prime numbers and conjectured that, for general \(v\), the number of unimodular equivalence classes of such simplices depends only on the divisor lattice of \(v\) \cite[Conjecture~4.5]{HibiTsuchiyaYoshida}.

The purpose of this paper is to prove this conjecture. Let \(D_v\) denote the divisor lattice of \(v\).
We first state the classification in a combinatorial form.
\begin{Theorem}\label{thm:intro-classification}
Let \(k\) and \(v\ge2\) be positive integers. The unimodular equivalence classes of Gorenstein simplices which are not lattice pyramids and whose \(h^*\)-polynomial is
\[
1+t^k+t^{2k}+\cdots+t^{(v-1)k}
\]
are in bijection with the following data:
\begin{enumerate}[(i)]
\item a strict divisor chain
\[
1=M_0<M_1<\cdots<M_s=v;
\]
\item for each \(i=1,\ldots,s\), a subset \(J_i\subseteq \Lc_{i-1}\), where
\[
\Lc_0=\emptyset,
\qquad
\Lc_i=(\Lc_{i-1}\setminus J_i)\cup\{\ell_i\},
\]
and \(\ell_i\) is a formal symbol created at the \(i\)-th step.
\end{enumerate}
\end{Theorem}

For a fixed strict divisor chain of length \(s\), the number of possible choices of the subsets \(J_i\) is \(s!\) (Proposition~\ref{prop:fixed-chain-count}). Therefore Theorem~\ref{thm:intro-classification} implies the following counting formula.

\begin{Theorem}\label{thm:intro-main}
Let \(k\) and \(v\ge2\) be positive integers, and let \(N(v,k)\) be the number of unimodular equivalence classes of Gorenstein simplices which are not lattice pyramids and whose \(h^*\)-polynomial is
\[
1+t^k+t^{2k}+\cdots+t^{(v-1)k}.
\]
Then
\[
N(v,k)=\sum_{s\ge1}c_s(D_v)s!,
\]
where \(c_s(D_v)\) denotes the number of strict divisor chains
\[
1=M_0<M_1<\cdots<M_s=v
\]
in \(D_v\). In particular, \(N(v,k)\) depends only on the divisor lattice \(D_v\).
\end{Theorem}

Theorem~\ref{thm:intro-classification} is constructive. From the combinatorial data appearing in the theorem, one can explicitly construct the corresponding finite subgroup, and hence the corresponding unimodular equivalence class of Gorenstein simplices through the finite subgroup model. This construction recovers the classification in \cite{HibiTsuchiyaYoshida,TsuchiyaGorsimp} when \(v\) is a prime number or the product of two prime numbers, and it also gives explicit classifications for $v=p^3$, where \(p\) is a prime number.

This paper is organized as follows. In Section~2, we recall the finite subgroup
model for lattice simplices. Section~3 studies one-step extensions. In
Section~4, we introduce a recursive construction and identify the admissible
extension data. Section~5 proves a block decomposition for groups of type
\((v,k)\), constructs a canonical quotient tower, and relates it to a subgroup
chain inside a concrete finite subgroup. In Section~6, we prove the
classification theorem and the counting formula. Finally, Section~7 illustrates
the classification by recovering the cases \(v=p\), \(v=p^2\), and \(v=pq\),
and by working out explicitly the case \(v=p^3\); it also derives general
counting formulas in the prime-power and squarefree cases.

\subsection*{Acknowledgment}
This work was supported by JSPS KAKENHI 22K13890 and 26K00618.

\section{Preliminaries}
In this section, we recall the finite subgroup model for lattice simplices and
introduce the notation used throughout the paper.
\subsection{Finite abelian groups associated to lattice simplices}
We begin by recalling the standard finite abelian group associated with a
lattice simplex.
Let
\[
\Delta=\conv(\vb_0,\dots,\vb_d)\subset \RR^d
\]
be a lattice simplex of dimension \(d\). We define a finite abelian group
\[
\Lambda_\Delta=
\left\{
(x_0,\dots,x_d)\in [0,1)^{d+1}:
\sum_{i=0}^d x_i(\vb_i,1)\in \ZZ^{d+1}
\right\},
\]
where the group operation is addition modulo \(1\) in each coordinate. For \(\xb=(x_0,\dots,x_d)\in \Lambda_\Delta\), we define its height by \(\hei(\xb):=\sum_{i=0}^d x_i\). Since \(\sum_{i=0}^d x_i(\vb_i,1)\in \ZZ^{d+1}\), the number \(\hei(\xb)\) is always an integer.

The importance of \(\Lambda_\Delta\) for us is that it encodes both the
\(h^*\)-polynomial of \(\Delta\) and the condition that \(\Delta\) is a lattice
pyramid. We recall these well-known facts.

\begin{Lemma}\label{lem:hstar-lambda}
Let \(\Delta\) be a lattice simplex of dimension \(d\), and write
\[
h^*(\Delta,t)=h_0^*+h_1^*t+\cdots+h_d^*t^d.
\]
Then, for any \(0\le i\le d\),
\[
h_i^*=|\{\xb\in \Lambda_\Delta:\hei(\xb)=i\}|.
\]
In particular,
\[
|\Lambda_\Delta|=h^*(\Delta,1)=\Vol(\Delta).
\]
\end{Lemma}

\begin{Lemma}[\cite{BatyrevHof}]\label{lem:pyramid-lambda}
Let \(\Delta\) be a lattice simplex of dimension \(d\). Then \(\Delta\) is a lattice pyramid if and only if there exists \(i\in\{0,\dots,d\}\) such that \(x_i=0\) for all \((x_0,\dots,x_d)\in \Lambda_\Delta\).
\end{Lemma}

We also recall that lattice simplices are classified by these finite abelian
subgroups.
\begin{Proposition}[\cite{BatyrevHof}]\label{prop:delta-lambda-correspondence}
Let \(d\ge1\). Then there is a bijection between unimodular equivalence classes of \(d\)-dimensional lattice simplices with ordered vertices and finite abelian subgroups
$G\subset [0,1)^{d+1}$
such that \(\sum_{i=0}^d x_i\in \ZZ\) for all \((x_0,\dots,x_d)\in G\). Consequently, unimodular equivalence classes of \(d\)-dimensional lattice simplices are in bijection with such finite abelian subgroups up to permutation of coordinates.
\end{Proposition}

\subsection{Groups of type \((v,k)\)}

Let \(H\) be a finite abelian group endowed with an integer-valued function \(\hei:H\to \ZZ_{\ge0}\). Let \(v\ge2\) and \(k\ge1\) be integers. We say that \(H\) is of \emph{type \((v,k)\)} if \(|H|=v\) and
\[
\{\hei(\xb):\xb\in H\}=\{0,k,2k,\ldots,(v-1)k\}.
\]
Equivalently, the heights of the elements of \(H\) are exactly \(0,k,2k,\ldots,(v-1)k\), each occurring once.
If \(G\subset[0,1)^N\) is a finite subgroup such that \(\hei(\xb)\in \ZZ\) for all \(\xb\in G\), then we regard \(G\) as a finite abelian group with height function \(\hei\). We say that the \(i\)-th coordinate of \(G\) is zero if \(x_i=0\) for all \(\xb=(x_1,\ldots,x_N)\in G\), and that \(G\) has no zero coordinate if no coordinate is zero in this sense.

From Lemmas~\ref{lem:hstar-lambda} and \ref{lem:pyramid-lambda}, together with Proposition~\ref{prop:delta-lambda-correspondence}, we obtain the following.

\begin{Proposition}\label{prop:type-vk-simplex}
Let \(v\ge2\) and \(k\ge1\). Under the correspondence in Proposition~\ref{prop:delta-lambda-correspondence}, the unimodular equivalence classes of lattice simplices which are not lattice pyramids and whose \(h^*\)-polynomial is
\[
1+t^k+t^{2k}+\cdots+t^{(v-1)k}
\]
correspond to coordinate-permutation equivalence classes of finite subgroups \(G\subset [0,1)^{d+1}\) of type \((v,k)\) with no zero coordinate.
\end{Proposition}

\subsection{Coordinate projections}

Let \(G\subset[0,1)^N\) be a finite subgroup. For each coordinate \(i\), consider the projection
\[
\pi_i:G\to [0,1),
\qquad
\xb=(x_1,\dots,x_N)\mapsto x_i,
\]
where \([0,1)\) is regarded as an abelian group under addition modulo \(1\). We define \(d_i(G):=|\im(\pi_i)|\).

\begin{Lemma}\label{lem:coordinate-values}
Let \(G\subset[0,1)^N\) be a finite subgroup of order \(M\), and put \(d_i=d_i(G)\). Then, as \(\xb=(x_1,\ldots,x_N)\) runs over \(G\), the \(i\)-th coordinate \(x_i\) takes the values
\[
0,\frac1{d_i},\frac2{d_i},\ldots,\frac{d_i-1}{d_i}
\]
each exactly \(M/d_i\) times.
\end{Lemma}

\begin{proof}
The image of \(\pi_i\) has order \(d_i\). Since \([0,1)\), endowed with addition modulo \(1\), has a unique subgroup of order \(d_i\), namely
\[
\left\{
0,\frac1{d_i},\frac2{d_i},\ldots,\frac{d_i-1}{d_i}
\right\},
\]
we have
\[
\im(\pi_i)=
\left\{
0,\frac1{d_i},\frac2{d_i},\ldots,\frac{d_i-1}{d_i}
\right\}.
\]
Moreover, every fiber of \(\pi_i\) has cardinality
\[
|\Ker(\pi_i)|=\frac{|G|}{|\im(\pi_i)|}=\frac{M}{d_i}.
\]
Thus each value occurs exactly \(M/d_i\) times.
\end{proof}

\section{One-step extensions}

In this section, we study how a finite subgroup can be enlarged by adjoining one
new generator. We first construct such extensions explicitly, and then prove a
converse.

From now on, we fix a finite subgroup \(G\subset [0,1)^N\) of type \((M,k)\)
with no zero coordinate, and an integer \(n\ge2\). We will study extensions
obtained by adjoining one element whose image generates a cyclic quotient of
order \(n\). We denote by \(C_n\) the cyclic group of order \(n\).

We first explain how to construct such an extension from a suitable element of
\(G\).
\begin{Lemma}\label{lem:construction-from-h}
Let \(G\subset [0,1)^N\) be of type \((M,k)\) and have no zero coordinate. Let
\(\hb\in G\) satisfy
\[
h_i\in\left\{0,\frac1{d_i(G)}\right\}
\qquad
(i=1,\ldots,N).
\]
Set \(L:=nMk-\hei(\hb)\). Embed \(G\) into \([0,1)^{N+L}\) by adding \(L\) zero
coordinates. We refer to the first \(N\) coordinates as the old coordinates and
to the last \(L\) coordinates as the new coordinates. Define an element
\(\cb\in [0,1)^{N+L}\) by
\[
c_i=
\begin{cases}
\dfrac{h_i}{n} & (i=1,\ldots,N),\\[2mm]
\dfrac1n & (i=N+1,\ldots,N+L).
\end{cases}
\]
Let \(G':=G+\langle \cb\rangle\). Then \(G'/G\cong C_n\), \(n\cb=\hb\), and \(G'\) is of type \((Mn,k)\). Moreover,
the heights in the coset \(r\cb+G\) are
\[
rMk,\ rMk+k,\ \ldots,\ rMk+(M-1)k
\]
for each \(r=0,1,\ldots,n-1\).
\end{Lemma}

\begin{proof}
Since \(G\) is of type \((M,k)\), every element of \(G\) has height at most \((M-1)k\). Hence \(0\le \hei(\hb)\le (M-1)k\), so \(L=nMk-\hei(\hb)>0\). The new \(1/n\)-coordinates ensure that the image of \(\cb\) in \(G'/G\) has order exactly \(n\). Hence \(G'/G\cong C_n\).

By construction,
\[
\hei(\cb)=\frac{\hei(\hb)}{n}+\frac{L}{n}=Mk.
\]
On the old coordinates, we have \(nc_i=h_i\), while on each new coordinate \(nc_i=1\equiv0\pmod 1\). Since \(\hb\) is embedded with zero new coordinates, it follows that \(n\cb=\hb\).

We claim that no carry occurs when adding \(r\cb\) to an element of \(G\), for \(0\le r\le n-1\). Viewing \(G\) as a subgroup of \([0,1)^{N+L}\) via this embedding, let
\(\xb=(x_1,\ldots,x_{N+L})\in G\). If \(h_i=0\), then \(c_i=0\). If \(h_i=1/d_i(G)\), then \(c_i=1/(nd_i(G))\). By Lemma~\ref{lem:coordinate-values}, the possible values of \(x_i\) are
\[
0,\frac1{d_i(G)},\ldots,\frac{d_i(G)-1}{d_i(G)},
\]
so \(x_i\le (d_i(G)-1)/d_i(G)\). Since \(r\le n-1\), we have \(rc_i\le (n-1)/(nd_i(G))\), and therefore
\[
x_i+rc_i
\le
\frac{d_i(G)-1}{d_i(G)}+\frac{n-1}{nd_i(G)}
=
1-\frac1{nd_i(G)}
<1.
\]
Thus no carry occurs on the old coordinates. On each new coordinate, \(x_i=0\) and \(rc_i=r/n<1\), so there is no carry there either.

Therefore \(\hei(\xb+r\cb)=\hei(\xb)+r\hei(\cb)=\hei(\xb)+rMk\). Since \(G\) is of type \((M,k)\), the heights in \(G\) are \(0,k,\ldots,(M-1)k\). Hence the heights in \(r\cb+G\) are exactly
\[
rMk,\ rMk+k,\ \ldots,\ rMk+(M-1)k.
\]
As \(r=0,1,\ldots,n-1\), these blocks form the full set of heights \(0,k,\ldots,(Mn-1)k\). Thus \(G'\) is of type \((Mn,k)\).
\end{proof}

To prove the converse direction, we need a simple calculation describing how the
sum of the fractional parts changes under a uniform shift.
Here and in what follows, for a real number \(x\), we write \(\{x\}\) for the fractional part of \(x\), that is, the unique element of \([0,1)\) congruent to \(x\) modulo \(1\).

\begin{Lemma}\label{lem:coordinate-sum-shift}
Let \(d,n\) be positive integers, and let
\[
\alpha=\frac{s}{nd}
\qquad
(0\le s<nd).
\]
Write \(s=qn+r\) with \(0\le r<n\). Then
\[
\sum_{j=0}^{d-1}
\left\{\frac{j}{d}+\alpha\right\}
-
\sum_{j=0}^{d-1}\frac{j}{d}
=
\frac{r}{n}.
\]
\end{Lemma}

\begin{proof}
Since
\[
\alpha=\frac{s}{nd}=\frac{qn+r}{nd}=\frac{q}{d}+\frac{r}{nd},
\]
we have
\[
\left\{\frac{j}{d}+\alpha\right\}
=
\left\{\frac{j+q}{d}+\frac{r}{nd}\right\}.
\]
As \(j\) runs through \(0,1,\ldots,d-1\), the residue class \(j+q\) modulo \(d\) also runs through all residue classes modulo \(d\). Hence adding \(q/d\) only permutes the set \(\{0,\frac1d,\ldots,\frac{d-1}{d}\}\) modulo \(1\). Therefore it is enough to compute the change caused by adding \(r/(nd)\).

Since \(0\le r/(nd)<1/d\), we have \(\frac{j}{d}+\frac{r}{nd}<1\) for all \(j=0,\ldots,d-1\). Thus no wrap-around occurs, and hence
\[\sum_{j=0}^{d-1}
\left\{\frac{j}{d}+\alpha\right\}
-
\sum_{j=0}^{d-1}\frac{j}{d}=
\sum_{j=0}^{d-1}
\left\{\frac{j}{d}+\frac{r}{nd}\right\}
-
\sum_{j=0}^{d-1}\frac{j}{d}
=
d\cdot\frac{r}{nd}
=
\frac{r}{n}.
\]
\end{proof}

We now prove that every extension of the above form arises from the
construction in Lemma~\ref{lem:construction-from-h}.

\begin{Lemma}\label{lem:converse-one-step}
Let \(G\subset [0,1)^N\) be of type \((M,k)\) and have no zero coordinate. Embed \(G\) into some \([0,1)^{N'}\) by adding zero coordinates, let \(\cb\in [0,1)^{N'}\), and set \(G':=G+\langle \cb\rangle\subset [0,1)^{N'}\). Assume that
\[
G'/G\cong C_n,
\qquad
G' \text{ is of type }(Mn,k),
\]
that the heights in the coset \(r\cb+G\) are
\[
rMk,\ rMk+k,\ \ldots,\ rMk+(M-1)k
\]
for each \(r=0,1,\ldots,n-1\), and that \(\cb\) is the element of minimum height in \(\cb+G\). Then \(\hb:=n\cb\) belongs to \(G\) and satisfies
\[
h_i\in\left\{0,\frac1{d_i(G)}\right\}
\qquad
(i=1,\ldots,N).
\]
Moreover, after deleting the coordinates on which every element of \(G'\) is zero, the resulting subgroup is obtained from \(\hb\) by the construction of Lemma~\ref{lem:construction-from-h}.
\end{Lemma}

\begin{proof}
Since \(G'/G\cong C_n\), we have \(n\cb\in G\). Put \(\hb:=n\cb\). For each coordinate \(i=1,\ldots,N'\), let \(d_i\) denote the order of the \(i\)-th coordinate projection of the embedded subgroup \(G\). Thus \(d_i=d_i(G)\) for \(i\le N\), while \(d_i=1\) for \(i>N\), since \(G\) is zero on those coordinates.

Since \(nc_i\) belongs to the image of the \(i\)-th coordinate projection of \(G\), we may write
\[
c_i=\frac{s_i}{nd_i}
\qquad
(0\le s_i<nd_i).
\]
Write \(s_i=q_in+r_i\) with \(0\le r_i<n\).

We compare the sum of heights of \(G\) and \(\cb+G\). For \(\xb=(x_1,\ldots,x_{N'})\in G\), one has
\[
\hei(\xb+\cb)=\sum_{i=1}^{N'}\{x_i+c_i\}.
\]
Hence 
\[
\sum_{\xb\in G}\hei(\xb+\cb)-\sum_{\xb\in G}\hei(\xb)
=
\sum_{i=1}^{N'}
\left(
\sum_{\xb\in G}\{x_i+c_i\}
-
\sum_{\xb\in G}x_i
\right).
\]
Fix \(i\). By Lemma~\ref{lem:coordinate-values}, each value
$0,\frac1{d_i},\ldots,\frac{d_i-1}{d_i}$
occurs as the \(i\)-th coordinate \(x_i\) of an element of \(G\) exactly \(M/d_i\) times. Therefore
\[
\sum_{\xb\in G}\{x_i+c_i\}-\sum_{\xb\in G}x_i
=
\frac{M}{d_i}
\left(
\sum_{j=0}^{d_i-1}
\left\{\frac{j}{d_i}+c_i\right\}
-
\sum_{j=0}^{d_i-1}\frac{j}{d_i}
\right).
\]
Since \(c_i=s_i/(nd_i)\) and \(s_i=q_in+r_i\), Lemma~\ref{lem:coordinate-sum-shift} gives
\[
\sum_{\xb\in G}\{x_i+c_i\}-\sum_{\xb\in G}x_i
=
\frac{M}{d_i}\frac{r_i}{n}
\qquad
(i=1,\ldots,N').
\]
By assumption, the heights in \(\cb+G\) are
\[
Mk,Mk+k,\ldots,Mk+(M-1)k,
\]
whereas the heights in \(G\) are
\[
0,k,\ldots,(M-1)k.
\]
Hence the difference of the two total height sums is
\[
\sum_{j=0}^{M-1}(Mk+jk)-\sum_{j=0}^{M-1}jk=M^2k.
\]
Therefore
\[
\sum_{i=1}^{N'} \frac{M}{d_i}\frac{r_i}{n}=M^2k.
\tag{1}
\]

On the other hand, since \(\cb\) is the minimum-height element of \(\cb+G\), we have \(\hei(\cb)=Mk\). Using \(c_i=s_i/(nd_i)\), we obtain
\[
\sum_{i=1}^{N'}\frac{M}{d_i}\frac{s_i}{n}
=
M\sum_{i=1}^{N'} c_i
=
M\hei(\cb)
=
M^2k.
\tag{2}
\]
Comparing (1) and (2), we obtain
\[
\sum_{i=1}^{N'} \frac{M}{d_i}\frac{s_i-r_i}{n}=0.
\]
Since \(s_i-r_i=q_in\ge0\), every summand is nonnegative. Hence every summand is zero, so \(q_i=0\) for every \(i\). Thus \(0\le s_i<n\) for all \(i\).

Next compare the coset \((n-1)\cb+G\) with \(G\). By assumption, its heights are
\[
(n-1)Mk,\ (n-1)Mk+k,\ \ldots,\ (n-1)Mk+(M-1)k,
\]
and therefore
\[
\sum_{\zb\in (n-1)\cb+G}\hei(\zb)-\sum_{\xb\in G}\hei(\xb)
=
(n-1)M^2k.
\tag{3}
\]
We now compute the left-hand side of \((3)\) by considering each coordinate separately.
For a fixed coordinate \(i\), the contribution to the total height difference is
\[
\sum_{\xb\in G}\{x_i+(n-1)c_i\}-\sum_{\xb\in G}x_i.
\]
If \(s_i=0\), then \(c_i=0\), so this contribution is \(0\). Suppose that
\(s_i>0\). Since \(0<s_i<n\), one has
\[
(n-1)c_i=\frac{(n-1)s_i}{nd_i}
=\frac{s_i}{d_i}-\frac{s_i}{nd_i}.
\]
Hence, for each \(j=0,1,\ldots,d_i-1\),
\[
\left\{\frac{j}{d_i}+(n-1)c_i\right\}
=
\left\{\frac{j+s_i}{d_i}-\frac{s_i}{nd_i}\right\}.
\]
As \(j\) runs through \(0,1,\ldots,d_i-1\), the residue class \(j+s_i\) modulo
\(d_i\) also runs through all residue classes modulo \(d_i\). Therefore
\[
\sum_{j=0}^{d_i-1}\left\{\frac{j}{d_i}+(n-1)c_i\right\}
=
\sum_{j=0}^{d_i-1}\left\{\frac{j}{d_i}-\frac{s_i}{nd_i}\right\}.
\]
Since
\[
-\frac{s_i}{nd_i}\equiv \frac{nd_i-s_i}{nd_i}\pmod 1
\]
and
\[
nd_i-s_i=(d_i-1)n+(n-s_i),
\]
Lemma~\ref{lem:coordinate-sum-shift} yields
\[
\sum_{j=0}^{d_i-1}\left\{\frac{j}{d_i}+(n-1)c_i\right\}
-
\sum_{j=0}^{d_i-1}\frac{j}{d_i}
=
\frac{n-s_i}{n}
=
1-\frac{s_i}{n}.
\]
Hence the contribution of the \(i\)-th coordinate is
\[
\frac{M}{d_i}\left(1-\frac{s_i}{n}\right).
\]
Summing over all coordinates, we obtain
\[
\sum_{s_i>0}\frac{M}{d_i}\left(1-\frac{s_i}{n}\right)
=
(n-1)M^2k.
\tag{4}
\]
Using (2), this becomes
\[
\sum_{s_i>0}\frac{M}{d_i}=nM^2k.
\tag{5}
\]
Multiplying (2) by \(n\), we also have
\[
\sum_{i=1}^{N'}\frac{M}{d_i}s_i=nM^2k.
\tag{6}
\]
Comparing (5) and (6), we obtain
\[
\sum_{s_i>0}\frac{M}{d_i}(s_i-1)=0.
\]
Again every summand is nonnegative, so \(s_i\in\{0,1\}\) for every \(i\).

For \(i=1,\ldots,N\), since \(\hb=n\cb\) in \([0,1)^{N'}\), we have
\[
h_i\equiv nc_i\equiv \frac{s_i}{d_i}\pmod 1.
\]
Since \(s_i\in\{0,1\}\), one has \(\frac{s_i}{d_i}\in [0,1)\). Hence
\[
h_i=\frac{s_i}{d_i}.
\]
Thus
\[
h_i\in\left\{0,\frac1{d_i}\right\}
=
\left\{0,\frac1{d_i(G)}\right\}.
\]
For \(i>N\), we have \(d_i=1\), so \(c_i=s_i/n\in\{0,1/n\}\). Hence, on every extra coordinate, the value of \(\cb\) is either \(0\) or \(1/n\). Therefore, after deleting the coordinates on which every element of \(G'\) is zero, the remaining extra coordinates are exactly new \(1/n\)-coordinates. It follows that the resulting subgroup is obtained from \(\hb\) by the construction of Lemma~\ref{lem:construction-from-h}.
\end{proof}

Combining the construction and the converse, we obtain the following
classification of one-step extensions.
\begin{Theorem}\label{thm:one-step-extension}
Let \(G\subset [0,1)^N\) be of type \((M,k)\) and have no zero coordinate, and
let \(n\ge2\). Then Lemmas~\ref{lem:construction-from-h} and
\ref{lem:converse-one-step} yield a bijection between the following data:
\begin{enumerate}[(i)]
\item elements \(\hb\in G\) satisfying
\[
h_i\in\left\{0,\frac1{d_i(G)}\right\}
\qquad
(i=1,\ldots,N);
\]
\item extensions obtained by embedding \(G\) into some \([0,1)^{N'}\) by adding zero coordinates and taking
\[
G'=G+\langle \cb\rangle \subset [0,1)^{N'}
\]
such that
$G'/G\cong C_n$,
$G'$ is of type $(Mn,k)$ and
the heights in the coset \(r\cb+G\) are
\[
rMk,\ rMk+k,\ \ldots,\ rMk+(M-1)k
\]
for each \(r=0,1,\ldots,n-1\), and \(\cb\) is the minimum-height element of \(\cb+G\), where extensions are identified after deleting the coordinates on which every element of \(G'\) is zero.
\end{enumerate}
\end{Theorem}

\section{Recursive construction}

In this section, we introduce a recursive construction of finite subgroups
together with associated finite subsets \(\Lc(G)\subset G\). The role of
\(\Lc(G)\) is to record the elements that can be used in the next step of the
construction. More precisely, we will show that, for every pair
\((G,\Lc(G))\) obtained in this way, the elements \(\hb\in G\) satisfying
\[
h_i\in\left\{0,\frac1{d_i(G)}\right\}
\qquad
(i=1,\ldots,N)
\]
are precisely the subset sums of elements of \(\Lc(G)\). By
Theorem~\ref{thm:one-step-extension}, these are exactly the elements that give
rise to one-step extensions of \(G\).

We now define recursively the pairs
$(G,\Lc(G))$.
First, for an integer \(m\ge2\), let
\[
\cb=
\left(\frac1m,\ldots,\frac1m\right)\in [0,1)^{mk}
\]
and set
$G=\langle \cb\rangle$.
Then \(G\) is of type \((m,k)\), and we define
$\Lc(G)=\{\cb\}$.

Now suppose that a pair \((G,\Lc(G))\) has already been constructed, where
\(G\subset [0,1)^N\) is of type \((M,k)\). Let \(J\subseteq \Lc(G)\) be a subset
for which the element
$\hb:=\sum_{\ub\in J}\ub\in G$
satisfies
\[
h_i\in\left\{0,\frac1{d_i(G)}\right\}
\qquad
(i=1,\ldots,N).
\]
Let \(n\ge2\), and apply Lemma~\ref{lem:construction-from-h} to \(\hb\). Thus
we obtain an extension
$G'=G+\langle \cb\rangle$
such that \(n\cb=\hb\). We then define
\[
\Lc(G')=(\Lc(G)\setminus J)\cup\{\cb\}.
\]

\begin{Lemma}\label{lem:L-supports}
Let \((G,\Lc(G))\) be one of the pairs constructed recursively above. Then the
supports of the elements of \(\Lc(G)\) are pairwise disjoint. Moreover, if
\(\ub\in\Lc(G)\) and \(i\in \supp(\ub)\), then
\[
u_i=\frac1{d_i(G)}.
\]
\end{Lemma}

\begin{proof}
We prove both assertions by induction on the recursive construction.

We first consider an initial pair arising from the first step of the
construction. Thus
\[
G=\langle \cb\rangle,
\qquad
\cb=
\left(\frac1m,\ldots,\frac1m\right)\in [0,1)^{mk},
\qquad
\Lc(G)=\{\cb\}
\]
for some \(m\ge2\). Since \(\Lc(G)\) consists of a single element, the support
statement is clear. Moreover, for each coordinate \(i\), the \(i\)-th coordinate
projection of \(G\) has image
\[
\left\{0,\frac1m,\ldots,\frac{m-1}{m}\right\},
\]
so \(d_i(G)=m\). Since every coordinate of \(\cb\) is equal to \(1/m\), the
second assertion holds.

Now assume that the assertions hold for a pair \((G,\Lc(G))\), and let
\((G',\Lc(G'))\) be obtained from it as above. Thus
\[
G'=G+\langle \cb\rangle,
\qquad
n\cb=\hb,
\qquad
\Lc(G')=(\Lc(G)\setminus J)\cup\{\cb\}.
\]

By the induction hypothesis, the supports of the elements of \(\Lc(G)\) are
pairwise disjoint. Hence the support of
$\hb=\sum_{\ub\in J}\ub$
is the disjoint union of the supports of the elements in \(J\). By the
construction in Lemma~\ref{lem:construction-from-h}, on the old coordinates one
has \(c_i=h_i/n\), while on each newly added coordinate one has \(c_i=1/n\).
Thus \(\supp(\cb)\) consists of the support of \(\hb\), together with the newly
added coordinates. In particular, \(\supp(\cb)\) is disjoint from the supports
of the elements in \(\Lc(G)\setminus J\). Therefore the supports of the elements
of \(\Lc(G')\) are pairwise disjoint.

We next prove the assertion about coordinate values. Let
\(\ub\in\Lc(G')\) and \(i\in\supp(\ub)\).
First suppose that \(\ub\in\Lc(G)\setminus J\). Since \(\ub\notin J\) and the
supports of the elements of \(\Lc(G)\) are pairwise disjoint, the \(i\)-th
coordinate of \(\hb\) is zero. Hence \(c_i=h_i/n=0\). Therefore adjoining
\(\cb\) does not change the image of the \(i\)-th coordinate projection, so
\(d_i(G')=d_i(G)\). By the induction hypothesis, \(u_i=1/d_i(G)\), and hence
\[
u_i=\frac1{d_i(G')}.
\]

It remains to consider the case \(\ub=\cb\). Let \(i\in\supp(\cb)\). If \(i\)
is a newly added coordinate, then \(c_i=1/n\). Since every element of \(G\) is
zero on that coordinate, the \(i\)-th coordinate projection of \(G'\) is
generated by \(1/n\). Thus \(d_i(G')=n\), and hence
\[
c_i=\frac1{d_i(G')}.
\]

Now suppose that \(i\) is an old coordinate. Since \(i\in\supp(\cb)\), one has
\(h_i\neq0\), and hence \(i\in\supp(\hb)\). By the disjointness of the supports
in \(\Lc(G)\), there is a unique element \(\ub_0\in J\) such that
\(i\in\supp(\ub_0)\). By the induction hypothesis,
\[
(u_0)_i=\frac1{d_i(G)}.
\]
Since the supports of the elements in \(J\) are pairwise disjoint, this implies
\[
h_i=\frac1{d_i(G)}.
\]
Therefore
\[
c_i=\frac{h_i}{n}=\frac1{n\,d_i(G)}.
\]
The image of the \(i\)-th coordinate projection of \(G\) is generated by
\(1/d_i(G)\), while after adjoining \(\cb\), the image of the \(i\)-th
coordinate projection of \(G'\) is generated by \(1/d_i(G)\) and
\(1/(n\,d_i(G))\). Hence it is generated by \(1/(n\,d_i(G))\), so
$d_i(G')=n\,d_i(G)$.
Thus
\[
c_i=\frac1{d_i(G')}.
\]
This proves the lemma.
\end{proof}

We now use this structural property of \(\Lc(G)\) to identify the elements of
\(G\) that satisfy the coordinate condition appearing in the one-step extension
theorem.

\begin{Theorem}\label{thm:subset-sums}
Let \((G,\Lc(G))\) be one of the pairs constructed recursively above. Then the
elements \(\hb\in G\) satisfying
\[
h_i\in\left\{0,\frac1{d_i(G)}\right\}
\qquad
(i=1,\ldots,N)
\]
are precisely the subset sums
\[
\sum_{\ub\in J}\ub,
\qquad
J\subseteq \Lc(G).
\]
Moreover, such a subset \(J\) is unique.
\end{Theorem}

\begin{proof}
We prove the assertion by induction on the recursive construction.

We first consider an initial pair. Thus \(G=\langle \cb\rangle\), where
$\cb=\left(\frac1m,\ldots,\frac1m\right)\in [0,1)^{mk}$
for some \(m\ge2\), and \(\Lc(G)=\{\cb\}\). Let \(\hb\in G\) satisfy
\[
h_i\in\left\{0,\frac1{d_i(G)}\right\}
\qquad
(i=1,\ldots,mk).
\]
Since \(G=\langle \cb\rangle\) is cyclic of order \(m\), there exists
\(t\in\{0,1,\ldots,m-1\}\) such that
$\hb=t\cb$.
Moreover, \(d_i(G)=m\) for every coordinate \(i\), and each coordinate of
\(\hb=t\cb\) is equal to \(t/m\). Therefore the coordinate condition becomes
\[
\frac{t}{m}\in\left\{0,\frac1m\right\},
\]
which forces \(t=0\) or \(t=1\). Hence \(\hb=0\) or \(\hb=\cb\). These are
precisely the subset sums of \(\Lc(G)=\{\cb\}\), and the representing subset is
clearly unique.

Now assume that the assertion holds for a pair \((G,\Lc(G))\), and let
\((G',\Lc(G'))\) be obtained from it as above. Thus
$G'=G+\langle \cb\rangle,
n\cb=\hb$ and 
$\Lc(G')=(\Lc(G)\setminus J)\cup\{\cb\}$.
We prove the assertion for \(G'\).
First, any subset sum of elements of \(\Lc(G')\) satisfies the coordinate
condition for \(G'\). Indeed, by Lemma~\ref{lem:L-supports}, the supports of the
elements of \(\Lc(G')\) are pairwise disjoint, and on each coordinate in the
support of an element \(\ub\in\Lc(G')\), the element \(\ub\) has value
\(1/d_i(G')\). Therefore a subset sum of elements of \(\Lc(G')\) has, in each
coordinate \(i\), either value \(0\) or value \(1/d_i(G')\).

Conversely, let \(\hb'\in G'\) satisfy
\[
h_i'\in\left\{0,\frac1{d_i(G')}\right\}
\qquad
(i=1,\ldots,N'),
\]
where \(N'\) is the number of coordinates of \(G'\). We show that \(\hb'\) is a
subset sum of elements of \(\Lc(G')\).

Since \(G'/G\cong C_n\), every element of \(G'\) can be written uniquely as
$\hb'=\gb+r\cb$
with \(\gb\in G\) and \(0\le r<n\). On each newly added coordinate, \(\gb\) has
value \(0\), while \(\cb\) has value \(1/n\). Since on such a coordinate
\(d_i(G')=n\), the coordinate condition on \(\hb'\) implies
\[
\frac rn\in\left\{0,\frac1n\right\}.
\]
Thus \(r=0\) or \(r=1\).

We claim that, in both cases, \(\gb\) is a subset sum of elements of
\(\Lc(G)\setminus J\). Let \(i\) be a coordinate in the support of an element of
\(J\). By Lemma~\ref{lem:L-supports}, the supports of the elements of
\(\Lc(G)\) are pairwise disjoint, and each element of \(\Lc(G)\) has value
\(1/d_i(G)\) on its support. Since
$\hb=\sum_{\ub\in J}\ub$,
it follows that
$h_i=1/{d_i(G)}$.
Hence \(c_i=h_i/n=1/(n\,d_i(G))\), and the image of the \(i\)-th coordinate
projection of \(G'\) is generated by \(1/d_i(G)\) and \(1/(n\,d_i(G))\). Thus
$d_i(G')=n\,d_i(G)$.

If \(r=0\), then \(\hb'=\gb\). Since \(\gb\in G\), the \(i\)-th coordinate of
\(\gb\) belongs to
\[
\left\{0,\frac1{d_i(G)},\ldots,\frac{d_i(G)-1}{d_i(G)}\right\}.
\]
On the other hand, the coordinate condition for \(G'\) gives
\[
g_i\in\left\{0,\frac1{n\,d_i(G)}\right\}.
\]
The intersection of these two sets is \(\{0\}\), so \(g_i=0\).
If \(r=1\), then \(\hb'=\gb+\cb\). Since \(c_i=1/d_i(G')\) and \(\hb'\)
satisfies the coordinate condition for \(G'\), the same intersection argument
again gives \(g_i=0\). Therefore, in both cases, \(\gb\) is zero on the supports
of all elements in \(J\).

Now let \(i\) be an old coordinate not contained in the support of any element
of \(J\). Then \(h_i=0\), and hence \(c_i=0\). Therefore \(d_i(G')=d_i(G)\),
and the coordinate condition for \(\hb'=\gb+r\cb\) is exactly the coordinate
condition for \(\gb\) on this coordinate. Combining this with the conclusion on
the supports of \(J\), we see that \(\gb\) satisfies
\[
g_i\in\left\{0,\frac1{d_i(G)}\right\}
\qquad
(i=1,\ldots,N).
\]
By the induction hypothesis, \(\gb\) is a subset sum of elements of \(\Lc(G)\).
Since \(\gb\) is zero on the supports of all elements in \(J\), and the supports
of the elements of \(\Lc(G)\) are pairwise disjoint, no element of \(J\) occurs
in this subset sum. Thus \(\gb\) is a subset sum of elements of
\(\Lc(G)\setminus J\).

If \(r=0\), then \(\hb'=\gb\), so \(\hb'\) is a subset sum of elements of
\(\Lc(G')\). If \(r=1\), then \(\hb'=\gb+\cb\), so \(\hb'\) is again a subset
sum of elements of
$(\Lc(G)\setminus J)\cup\{\cb\}=\Lc(G')$.

Finally, uniqueness follows from Lemma~\ref{lem:L-supports}. Since the supports
of the elements of \(\Lc(G')\) are pairwise disjoint and each element is nonzero
on its support, a subset sum determines exactly which elements of \(\Lc(G')\)
occur in it.
\end{proof}

As an immediate consequence, for every recursively constructed pair
\((G,\Lc(G))\), every subset \(J\subseteq \Lc(G)\) gives rise to an element
$\hb=\sum_{\ub\in J}\ub$
satisfying
\[
h_i\in\left\{0,\frac1{d_i(G)}\right\}
\qquad
(i=1,\ldots,N).
\]
Hence, once a pair \((G,\Lc(G))\) has been constructed, every subset
\(J\subseteq \Lc(G)\) may be used in the next recursive step.

\section{First blocks and quotient towers}

In this section, we study finite abelian groups of type \((v,k)\) endowed with a
subadditive height function. The goal is to decompose such a group by repeatedly
factoring out the subgroup generated by the unique element of height \(k\). We
first show that this subgroup determines the first block of heights. We then
iterate this construction to obtain a canonical quotient tower. Finally, for a
concrete subgroup \(G\subset[0,1)^N\), we realize this quotient tower as a
chain of subgroups of \(G\).

We begin with a simple observation on cyclic sequences of integers. It will be
used to show that the set of heights on a coset forms a consecutive block.

\begin{Lemma}\label{lem:cyclic-walk}
Let
$r_0,r_1,\ldots,r_{m-1}$
be pairwise distinct integers, and set \(r_m:=r_0\). Assume that
\[
r_{t+1}-r_t\le 1
\qquad
(t=0,1,\ldots,m-1).
\]
Then there exists an integer \(L\) such that
\[
\{r_0,r_1,\ldots,r_{m-1}\}=\{L,L+1,\ldots,L+m-1\}.
\]
\end{Lemma}

\begin{proof}
Let \(L:=\min_t r_t\) and \(U:=\max_t r_t\). Choose an index \(a\) with
\(r_a=L\), and move cyclically from \(a\) until one first reaches an index
\(b\) with \(r_b=U\). Since every upward jump is at most \(1\), the sequence
cannot jump over any integer between \(L\) and \(U\). Hence every integer
\(L,L+1,\ldots,U\) appears among the \(r_t\). Since the \(r_t\) are pairwise
distinct and there are \(m\) of them, we have \(U-L+1=m\). Thus
\[
\{r_0,r_1,\ldots,r_{m-1}\}=\{L,L+1,\ldots,L+m-1\}.
\]
\end{proof}

We say that an integer-valued function \(\hei\) on an abelian group \(H\) is
\emph{subadditive} if
\[
\hei(u+v)\le \hei(u)+\hei(v)
\qquad
(u,v\in H).
\]
We will also need the following elementary fact on quotient heights.

\begin{Lemma}\label{lem:quotient-height-subadditive}
Let \(H\) be a finite abelian group endowed with an integer-valued subadditive
function \(\hei:H\to \ZZ_{\ge0}\). Let \(a\in H\), and put \(A:=\langle
a\rangle\). Assume that for every coset \(\overline{y}=y+A\in H/A\), the
minimum
\[
\min\{\hei(z):z\in \overline{y}\}
\]
is divisible by \(|A|\). Define
\[
\widetilde{\hei}(\overline{y})
:=
\frac1{|A|}\min\{\hei(z):z\in \overline{y}\}.
\]
Then \(\widetilde{\hei}\) is an integer-valued subadditive function on \(H/A\).
\end{Lemma}

\begin{proof}
Choose \(u\in \overline{x}\) and \(v\in \overline{y}\) such that
\[
\hei(u)=|A|\,\widetilde{\hei}(\overline{x}),
\qquad
\hei(v)=|A|\,\widetilde{\hei}(\overline{y}).
\]
Then \(u+v\in \overline{x}+\overline{y}\), so
\[
|A|\,\widetilde{\hei}(\overline{x}+\overline{y})
\le
\hei(u+v)
\le
\hei(u)+\hei(v)
=
|A|\,\widetilde{\hei}(\overline{x})
+
|A|\,\widetilde{\hei}(\overline{y}).
\]
Dividing by \(|A|\) gives the result.
\end{proof}

We now apply subadditivity to a group of type \((v,k)\). The first consequence
is that the powers of the unique element of height \(k\) have the expected
heights.

\begin{Lemma}\label{lem:powers-of-height-k}
Let \(H\) be a finite abelian group endowed with an integer-valued subadditive
function \(\hei\). Assume that \(H\) is of type \((v,k)\), and let \(a\in H\)
be the unique element of height \(k\). If \(m:=\ord(a)\), then
\[
\hei(ta)=tk
\qquad
(t=0,1,\ldots,m-1).
\]
\end{Lemma}

\begin{proof}
By subadditivity,
\[
\hei(ta)\le t\,\hei(a)=tk.
\]
Since \(H\) is of type \((v,k)\), every height is a multiple of \(k\). Hence
\[
\hei(ta)\in \{0,k,2k,\ldots,tk\}.
\]
For \(t=0,1,\ldots,m-1\), the elements \(0,a,\ldots,ta\) are pairwise distinct,
so their heights are pairwise distinct. Therefore
\[
\{\hei(0),\hei(a),\ldots,\hei(ta)\}
=
\{0,k,2k,\ldots,tk\},
\]
and in particular \(\hei(ta)=tk\).
\end{proof}

The next lemma is the main step of this section. It shows that the subgroup
generated by the unique element of height \(k\) cuts the height set into
consecutive blocks.

\begin{Lemma}\label{lem:abstract-first-block}
Let \(H\) be a finite abelian group endowed with an integer-valued subadditive
function \(\hei\). Assume that \(H\) is of type \((v,k)\), and let \(a\in H\)
be the unique element of height \(k\). Set \(m:=\ord(a)\) and
\(A:=\langle a\rangle\). Then the set of heights on every coset of \(A\) in
\(H\) is a consecutive block of length \(m\). More precisely, for every
\(y\in H\), there exists an integer \(q(y)\ge0\) such that
\[
\{\hei(y+ta):t=0,1,\ldots,m-1\}
=
\{q(y)mk,q(y)mk+k,\ldots,q(y)mk+(m-1)k\}.
\]
Consequently, if
\[
\widetilde{\hei}(y+A):=\frac1m\min\{\hei(z):z\in y+A\},
\]
then \(H/A\), endowed with \(\widetilde{\hei}\), is of type \((v/m,k)\), and
\(\widetilde{\hei}\) is subadditive.
\end{Lemma}

\begin{proof}
By Lemma~\ref{lem:powers-of-height-k},
\[
\hei(ta)=tk
\qquad
(t=0,1,\ldots,m-1).
\]
Fix \(y\in H\), and set
\[
r_t:=\frac1k\hei(y+ta)
\qquad
(t=0,1,\ldots,m-1),
\]
with \(r_m:=r_0\). Since \(a\) has order \(m\), the elements
\(y,y+a,\ldots,y+(m-1)a\) are pairwise distinct, hence their heights are
pairwise distinct. Therefore \(r_0,\ldots,r_{m-1}\) are pairwise distinct
integers.

By subadditivity,
\[
\hei(y+(t+1)a)
=
\hei((y+ta)+a)
\le
\hei(y+ta)+\hei(a)
=
\hei(y+ta)+k.
\]
Hence \(r_{t+1}-r_t\le 1\). By Lemma~\ref{lem:cyclic-walk}, there exists an
integer \(r\ge0\) such that
\[
\{r_0,\ldots,r_{m-1}\}=\{r,r+1,\ldots,r+m-1\}.
\]
Therefore
\[
\{\hei(y+ta):t=0,1,\ldots,m-1\}
=
\{rk,(r+1)k,\ldots,(r+m-1)k\}.
\]

The cosets of \(A\) partition \(H\), and the full height set of \(H\) is
\[
\{0,k,2k,\ldots,(v-1)k\}.
\]
These intervals are pairwise disjoint, since distinct elements of \(H\) have
distinct heights. Since they have length \(m\) and their union is
\(\{0,1,\ldots,v-1\}\), they must be
\[
\{0,\ldots,m-1\},\ \{m,\ldots,2m-1\},\ \ldots,\ \{v-m,\ldots,v-1\}.
\]
Hence the starting points are exactly
\[
0,m,2m,\ldots,\left(\frac vm-1\right)m.
\]
Therefore \(r=q(y)m\) for some \(q(y)\ge0\). This proves the block statement.

Now define
\[
\widetilde{\hei}(y+A)
=
\frac1m\min\{\hei(z):z\in y+A\}
=
q(y)k.
\]
As \(y+A\) runs through \(H/A\), the values \(q(y)\) are exactly
\[
0,1,\ldots,\frac vm-1.
\]
Hence \(H/A\), endowed with \(\widetilde{\hei}\), is of type \((v/m,k)\).
Subadditivity follows from Lemma~\ref{lem:quotient-height-subadditive}.
\end{proof}

We now iterate the first block construction.

\begin{Definition}\label{def:canonical-quotient-tower}
Let \(H\) be a finite abelian group of type \((v,k)\), endowed with an
integer-valued subadditive height function. Define recursively
\[
H_0:=H.
\]
If \(H_{i-1}\neq 0\), let \(a_i\in H_{i-1}\) be the unique element of height
\(k\), and put
\[
A_i:=\langle a_i\rangle,
\qquad
H_i:=H_{i-1}/A_i,
\]
where \(H_i\) is endowed with the quotient height from
Lemma~\ref{lem:abstract-first-block}. Since \(|H_i|<|H_{i-1}|\), this process
terminates after finitely many steps. We call the resulting sequence
\[
H=H_0,\ H_1,\ \ldots,\ H_s=0
\]
the canonical quotient tower of \(H\).
\end{Definition}

The quotient tower is defined abstractly. We next explain how, when the original
group is a concrete subgroup of \([0,1)^N\), this tower gives rise to a
canonical chain of subgroups inside the original group.

Before realizing the quotient tower for a concrete subgroup
\(G\subset[0,1)^N\), we recall the following structural fact about the unique
element of height \(k\) in a subgroup of type \((v,k)\).

\begin{Lemma}[{\cite[Lemma 3.5]{HibiTsuchiyaYoshida}}]\label{lem:height-k-rigidity}
Let \(G\subset[0,1)^N\) be a finite subgroup of type \((v,k)\), and let
\(\ab\in G\) be the unique element of height \(k\). Set \(m:=\ord(\ab)\), the
order of \(\ab\) in \(G\). Then, after permuting coordinates,
\[
\ab=
(\underbrace{1/m,\ldots,1/m}_{mk},0,\ldots,0).
\]
In particular,
\[
\hei(t\ab)=tk
\qquad
(t=0,1,\ldots,m-1).
\]
\end{Lemma}

We now realize the canonical quotient tower inside \(G\) as a chain of
subgroups. The previous lemma will later allow us to describe the first step of
this chain explicitly.
\begin{Lemma}\label{lem:kernel-tower}
Let \(G\subset[0,1)^N\) be a finite subgroup of type \((v,k)\) with no zero
coordinate. Regard \(G\) as a finite abelian group endowed with the height
function \(\hei\), and let
\[
G=H_0,\ H_1,\ \ldots,\ H_s=0
\]
be its canonical quotient tower. Let
\[
M_i:=|A_1|\cdots |A_i|
\qquad
(i\ge1),
\]
with \(M_0:=1\). For each \(i\), let \(\rho_i:G\to H_i\) be the composite of
the quotient maps from \(H_0=G\) to \(H_i\), and define
$G_i:=\Ker(\rho_i)$.
Then the following hold:
\begin{enumerate}[(i)]
\item \(G/G_i\cong H_i\) for each \(i\);
\item \(G_i/G_{i-1}\cong A_i\) for each \(i\ge1\); in particular, if
\(n_i:=|A_i|\), then \(|G_i/G_{i-1}|=n_i\) and \(M_i=M_{i-1}n_i\);
\item if \(\overline{\cb}_i\in G/G_{i-1}\) corresponds to \(a_i\in H_{i-1}\),
and if \(\cb_i\in G_i\) is the unique element of minimum height in the coset
\(\overline{\cb}_i\), then
\[
G_i=G_{i-1}+\langle \cb_i\rangle
\qquad\text{and}\qquad
n_i\cb_i\in G_{i-1}.
\]
\end{enumerate}
\end{Lemma}

\begin{proof}
For each \(i\), the homomorphism \(\rho_i:G\to H_i\) is surjective and, by
definition,
$G_i=\Ker(\rho_i)$.
Hence the first isomorphism theorem gives
$G/G_i\cong H_i$.
This proves (i).

Now fix \(i\ge1\). Since
$H_i=H_{i-1}/A_i$,
the kernel of the quotient map \(H_{i-1}\to H_i\) is \(A_i\). Since
\(G_{i-1}=\Ker(\rho_{i-1})\), the map \(\rho_i\) induces a homomorphism
$\overline{\rho}_i:G/G_{i-1}\to H_i$.
Its kernel is
\[
\Ker(\overline{\rho}_i)
=
\{x+G_{i-1}:x\in \Ker(\rho_i)\}
=
G_i/G_{i-1},
\]
since \(\Ker(\rho_i)=G_i\). Under the canonical isomorphism
$G/G_{i-1}\cong H_{i-1}$,
this kernel corresponds to \(A_i\). Hence
$G_i/G_{i-1}\cong A_i$.
Therefore \(|G_i/G_{i-1}|=|A_i|=n_i\), and
$M_i=M_{i-1}n_i$.
This proves (ii).

Let \(\overline{\cb}_i\in G/G_{i-1}\) be the element corresponding to
\(a_i\in H_{i-1}\) under the canonical isomorphism
$G/G_{i-1}\cong H_{i-1}$.
By (ii), the subgroup \(G_i/G_{i-1}\subset G/G_{i-1}\) corresponds to
\(A_i=\langle a_i\rangle\subset H_{i-1}\). Therefore
$G_i/G_{i-1}=\langle \overline{\cb}_i\rangle$.
Now let \(\cb_i\in G_i\) be any lift of \(\overline{\cb}_i\). Since the image of
\(\langle \cb_i\rangle\) in \(G/G_{i-1}\) is \(\langle \overline{\cb}_i\rangle\),
we obtain
$G_i=G_{i-1}+\langle \cb_i\rangle$.
Choosing the lift \(\cb_i\) of minimum height gives the stated element
\(\cb_i\). Finally, since \(\overline{\cb}_i\) has order \(n_i\) in
\(G/G_{i-1}\), one has
$n_i\cb_i\in G_{i-1}$.
This proves (iii).
\end{proof}

\begin{Lemma}\label{lem:concrete-quotient-height}
In the notation of Lemma~\ref{lem:kernel-tower}, for each \(i\ge0\), the
canonical isomorphism
$G/G_i\cong H_i$
is height-preserving in the following sense. If \(\overline{\xb}\in G/G_i\)
corresponds to \(\eta\in H_i\), then
\[
\hei_{H_i}(\eta)
=
\frac1{M_i}\min\{\hei_G(z):z\in \overline{\xb}\},
\]
where \(M_i=|G_i|\), \(\hei_G\) denotes the ordinary height on \(G\), and
\(\hei_{H_i}\) denotes the height on \(H_i\).
\end{Lemma}

\begin{proof}
We argue by induction on \(i\). For \(i=0\), one has \(M_0=1\), \(G_0=0\), and
\(H_0=G\), so the statement is tautological.

Assume that the assertion holds for \(i-1\). Let
\[
\pi_i:H_{i-1}\to H_i=H_{i-1}/A_i
\]
and
$\tau_i:G/G_{i-1}\to G/G_i$
be the natural quotient maps. The canonical isomorphisms
\[
G/G_{i-1}\cong H_{i-1},
\qquad
G/G_i\cong H_i
\]
are compatible with these quotient maps, as shown in the commutative diagram
\[
\begin{CD}
G/G_{i-1} @>{\tau_i}>> G/G_i \\
@V{\simeq}VV            @VV{\simeq}V \\
H_{i-1} @>{\pi_i}>> H_i .
\end{CD}
\]
Let \(\overline{\xb}\in G/G_i\), and let \(\eta\in H_i\) be the corresponding
element. By the commutativity of the diagram,
\(\tau_i^{-1}(\overline{\xb})\) corresponds to \(\pi_i^{-1}(\eta)\). Moreover,
regarding cosets as subsets of \(G\), the coset \(\overline{\xb}\) is
partitioned by the cosets in \(\tau_i^{-1}(\overline{\xb})\):
\[
\overline{\xb}
=
\bigsqcup_{\overline{\yb}\in \tau_i^{-1}(\overline{\xb})}
\overline{\yb}.
\]
By definition of the quotient height on \(H_i\),
\[
\hei_{H_i}(\eta)
=
\frac1{|A_i|}
\min\{\hei_{H_{i-1}}(u):u\in \pi_i^{-1}(\eta)\}.
\]
Using the correspondence between \(\tau_i^{-1}(\overline{\xb})\) and
\(\pi_i^{-1}(\eta)\), together with the induction hypothesis, we obtain
\[
\hei_{H_i}(\eta)
=
\frac1{|A_i|}
\min_{\overline{\yb}\in \tau_i^{-1}(\overline{\xb})}
\left\{
\frac1{M_{i-1}}
\min\{\hei_G(z):z\in \overline{\yb}\}
\right\}.
\]
Since
\[
\overline{\xb}
=
\bigsqcup_{\overline{\yb}\in \tau_i^{-1}(\overline{\xb})}
\overline{\yb}
\]
as subsets of \(G\), we have
\[
\min_{\overline{\yb}\in \tau_i^{-1}(\overline{\xb})}
\min\{\hei_G(z):z\in \overline{\yb}\}
=
\min\{\hei_G(z):z\in \overline{\xb}\}.
\]
Hence
\[
\hei_{H_i}(\eta)
=
\frac1{|A_i|}\cdot\frac1{M_{i-1}}
\min\{\hei_G(z):z\in \overline{\xb}\}.
\]
Since \(M_i=M_{i-1}|A_i|\), this is
\[
\hei_{H_i}(\eta)
=
\frac1{M_i}\min\{\hei_G(z):z\in \overline{\xb}\}.
\]
This proves the assertion.
\end{proof}

\begin{Lemma}\label{lem:Gi-type-and-block}
In the notation of Lemma~\ref{lem:kernel-tower}, each \(G_i\) is of type
\((M_i,k)\). Moreover, for each \(i\ge1\), if
\[
n_i=\frac{M_i}{M_{i-1}},
\]
then the heights in the coset \(r\cb_i+G_{i-1}\) are exactly
\[
rM_{i-1}k,\ rM_{i-1}k+k,\ \ldots,\ rM_{i-1}k+(M_{i-1}-1)k
\]
for each \(r=0,1,\ldots,n_i-1\).
\end{Lemma}

\begin{proof}
We argue by induction on \(i\). For \(i=1\), the subgroup
\(G_1=\langle \cb_1\rangle\) is cyclic of order \(M_1\), and \(\cb_1\) is the
unique element of height \(k\). By Lemma~\ref{lem:height-k-rigidity},
\[
\hei(t\cb_1)=tk
\qquad
(t=0,1,\ldots,M_1-1),
\]
so \(G_1\) is of type \((M_1,k)\).

Assume now that \(G_{i-1}\) is of type \((M_{i-1},k)\), and let
\(n_i=M_i/M_{i-1}\). By Lemma~\ref{lem:concrete-quotient-height}, the element
\[
r\cb_i+G_{i-1}\in G/G_{i-1}
\]
has quotient height \(rk\), because it corresponds to \(ra_i\in H_{i-1}\), whose
height is \(rk\) by Lemma~\ref{lem:powers-of-height-k}. Therefore
\[
\min\{\hei_G(z):z\in r\cb_i+G_{i-1}\}
=
M_{i-1}\cdot rk
=
rM_{i-1}k.
\]
Choose an element \(y_r\in r\cb_i+G_{i-1}\) such that
\(\hei_G(y_r)=rM_{i-1}k\).

Now let \(g\in G_{i-1}\). Since \(G_{i-1}\) is of type \((M_{i-1},k)\), one has
\(\hei_G(g)\in\{0,k,\ldots,(M_{i-1}-1)k\}\). The ordinary height on
\(G\subset [0,1)^N\) is subadditive, so
\[
\hei_G(y_r+g)\le \hei_G(y_r)+\hei_G(g)\le rM_{i-1}k+(M_{i-1}-1)k.
\]
Thus every height in the coset \(r\cb_i+G_{i-1}\) belongs to
\[
\{rM_{i-1}k,\ rM_{i-1}k+k,\ \ldots,\ rM_{i-1}k+(M_{i-1}-1)k\}.
\]
The coset \(r\cb_i+G_{i-1}\) has exactly \(M_{i-1}\) elements. Since \(G\) is of
type \((v,k)\), distinct elements of \(G\) have distinct heights. Hence the
heights in this coset are exactly the \(M_{i-1}\) values displayed above.

Taking the union over \(r=0,1,\ldots,n_i-1\), we see that the heights in \(G_i\)
are exactly
\[
0,k,2k,\ldots,(M_i-1)k.
\]
Thus \(G_i\) is of type \((M_i,k)\).
\end{proof}

\section{Classification and counting}

We now prove the classification theorem for finite subgroups of type \((v,k)\)
and derive the counting formula.

\begin{Theorem}\label{thm:tower-classification}
Coordinate-permutation equivalence classes of finite subgroups
$G\subset[0,1)^N$
of type \((v,k)\) with no zero coordinate are in bijection with the following data:
\begin{enumerate}[(i)]
\item a strict divisor chain
\[
1=M_0<M_1<\cdots<M_s=v;
\]
\item for each \(i=1,\ldots,s\), a subset
$J_i\subseteq \Lc_{i-1}$,
where the finite sets \(\Lc_i\) are recursively defined by
\[
\Lc_0=\emptyset,
\qquad
\Lc_i=(\Lc_{i-1}\setminus J_i)\cup\{\ell_i\},
\]
and \(\ell_i\) is a formal symbol created at the \(i\)-th step.
\end{enumerate}
\end{Theorem}

\begin{proof}
We construct maps in both directions.

First, suppose that combinatorial data as in the statement are given. Put
\[
n_i:=\frac{M_i}{M_{i-1}}
\qquad
(i=1,\ldots,s).
\]
We construct finite subgroups
\[
0=G_0\subset G_1\subset\cdots\subset G_s
\]
such that \(G_i\) is of type \((M_i,k)\), together with identifications
\[
\ell_j\longleftrightarrow \cb_j
\qquad
(j=1,\ldots,i),
\]
under which \(\Lc_i\) corresponds to \(\Lc(G_i)\).

At the first step, since \(\Lc_0=\emptyset\), one has \(J_1=\emptyset\). We
define
\[
\cb_1=
\left(\frac1{M_1},\ldots,\frac1{M_1}\right)\in [0,1)^{M_1k},
\qquad
G_1=\langle \cb_1\rangle.
\]
Then \(G_1\) is of type \((M_1,k)\), and by the initial step of the recursive
construction one has \(\Lc(G_1)=\{\cb_1\}\). Identifying
\(\ell_1\) with \(\cb_1\), the formal set
$\Lc_1=(\Lc_0\setminus J_1)\cup\{\ell_1\}$
corresponds to \(\Lc(G_1)\).

Now suppose that \(G_{i-1}\) has already been constructed, and that the formal
symbols in \(\Lc_{i-1}\) have been identified with the corresponding elements of
\(\Lc(G_{i-1})\). Then the chosen subset
$J_i\subseteq \Lc_{i-1}$
determines a subset of \(\Lc(G_{i-1})\), which we denote again by \(J_i\). Set
$\hb_i:=\sum_{\ub\in J_i}\ub\in G_{i-1}$.
By Theorem~\ref{thm:subset-sums}, \(\hb_i\) satisfies
\[
(h_i)_r\in\left\{0,\frac1{d_r(G_{i-1})}\right\}
\]
for every coordinate \(r\). Hence Lemma~\ref{lem:construction-from-h} yields
$G_i=G_{i-1}+\langle \cb_i\rangle$
such that
$G_i/G_{i-1}\cong C_{n_i}$,
$G_i$ is of type $(M_i,k)$ and
$n_i\cb_i=\hb_i$.
Moreover,
$\Lc(G_i)=(\Lc(G_{i-1})\setminus J_i)\cup\{\cb_i\}$.
Identifying \(\ell_i\) with \(\cb_i\), we see that the formal set
$\Lc_i=(\Lc_{i-1}\setminus J_i)\cup\{\ell_i\}$
corresponds to \(\Lc(G_i)\).

Proceeding inductively, we obtain a finite subgroup \(G_s\) of type \((v,k)\).
Moreover, \(G_s\) has no zero coordinate, because this is clear for \(G_1\), and
each recursive step only adds new nonzero coordinates while preserving the old
ones. Thus the construction determines a coordinate-permutation equivalence
class of finite subgroups of type \((v,k)\) with no zero coordinate.

Conversely, let \(G\subset[0,1)^N\) be a finite subgroup of type \((v,k)\) with
no zero coordinate. We now construct the combinatorial data appearing in the
statement.

Consider the canonical quotient tower
\[
G=H_0,\ H_1,\ \ldots,\ H_s=0.
\]
For each \(i\ge1\), let \(a_i\in H_{i-1}\) be the unique element of height
\(k\), and set
\[
A_i:=\langle a_i\rangle,
\qquad
M_i:=|A_1|\cdots |A_i|,
\qquad
n_i:=|A_i|.
\]
Let
\[
0=G_0\subset G_1\subset\cdots\subset G_s=G
\]
be the associated subgroup chain from Lemma~\ref{lem:kernel-tower}. Then
\[
1=M_0<M_1<\cdots<M_s=v
\]
is a strict divisor chain.

For each \(i\), let \(S_i\) be the set of coordinates that are nonzero on
\(G_i\), and let \(G_i^\red\) be obtained from \(G_i\) by deleting the
coordinates outside \(S_i\). Since \(G_{i-1}\subseteq G_i\), one has
\(S_{i-1}\subseteq S_i\). Hence, after adding zero coordinates indexed by
\(S_i\setminus S_{i-1}\), we may regard \(G_{i-1}^\red\) as a subgroup of
\(G_i^\red\). Deleting zero coordinates does not change heights. Hence each \(G_i^\red\) is
again of type \((M_i,k)\), and by construction it has no zero coordinate.

We define, by induction on \(i\), a finite subset
$\Lc(G_i^\red)\subset G_i^\red$
so that the pair
$(G_i^\red,\Lc(G_i^\red))$
is obtained by the recursive construction of Section~4.

For \(i=1\), the subgroup \(G_1\) is generated by the unique element
\(\cb_1\in G\) of height \(k\), whose order is \(M_1\). By
Lemma~\ref{lem:height-k-rigidity}, after permuting coordinates one has
\[
\cb_1=
\left(\frac1{M_1},\ldots,\frac1{M_1},0,\ldots,0\right),
\]
with exactly \(M_1k\) coordinates equal to \(1/M_1\). Hence
\[
G_1^\red=
\left\langle
\left(\frac1{M_1},\ldots,\frac1{M_1}\right)
\right\rangle.
\]
Thus \(G_1^\red\) is an initial group in the recursive construction, and we
define
$\Lc(G_1^\red):=\{\cb_1^\red\}$,
where
\[
\cb_1^\red=
\left(\frac1{M_1},\ldots,\frac1{M_1}\right).
\]

Now assume that \(i\ge2\), and that the pair
$(G_{i-1}^\red,\Lc(G_{i-1}^\red))$
has already been defined and obtained by the recursive construction. By
Lemma~\ref{lem:kernel-tower},
\[
G_i=G_{i-1}+\langle \cb_i\rangle
\qquad\text{and}\qquad
n_i\cb_i\in G_{i-1}.
\]
Set
$\hb_i:=n_i\cb_i\in G_{i-1}$,
let \(\hb_i^\red\in G_{i-1}^\red\) be the image of \(\hb_i\) after deleting the
zero coordinates of \(G_{i-1}\), and let \(\cb_i^\red\in G_i^\red\) be the
image of \(\cb_i\) after deleting the zero coordinates of \(G_i\).

By Lemma~\ref{lem:Gi-type-and-block}, \(G_{i-1}\) is of type \((M_{i-1},k)\),
\(G_i\) is of type \((M_i,k)\), and for each \(r=0,1,\ldots,n_i-1\), the heights
in the coset \(r\cb_i+G_{i-1}\) are
\[
rM_{i-1}k,\ rM_{i-1}k+k,\ \ldots,\ rM_{i-1}k+(M_{i-1}-1)k.
\]
Deleting zero coordinates does not change the quotient order \(n_i\), nor these
height sets. The image \(\cb_i^\red\) of \(\cb_i\) in \(G_i^\red\) satisfies
\[
G_i^\red=G_{i-1}^\red+\langle \cb_i^\red\rangle
\qquad\text{and}\qquad
G_i^\red/G_{i-1}^\red\cong C_{n_i}.
\]
Moreover, since deleting zero coordinates does not change heights,
\(\cb_i^\red\) is the minimum-height element of
\(\cb_i^\red+G_{i-1}^\red\). Thus all hypotheses of Theorem~\ref{thm:one-step-extension} are satisfied for
the extension \(G_{i-1}^\red\subset G_i^\red\).
Therefore \(\hb_i^\red\) satisfies
\[
(h_i)_r\in\left\{0,\frac1{d_r(G_{i-1}^\red)}\right\}
\]
for every coordinate \(r\). By Theorem~\ref{thm:subset-sums}, there is a unique
subset
$J_i\subseteq \Lc(G_{i-1}^\red)$
such that
$\hb_i^\red=\sum_{\ub\in J_i}\ub$.
We then define
\[
\Lc(G_i^\red)
:=
\bigl(\Lc(G_{i-1}^\red)\setminus J_i\bigr)\cup\{\cb_i^\red\}.
\]
By construction, the pair
$(G_i^\red,\Lc(G_i^\red))$
is obtained from
$(G_{i-1}^\red,\Lc(G_{i-1}^\red))$
by one recursive step.

Thus, for each \(i\ge1\), we obtain a uniquely determined subset
$J_i\subseteq \Lc(G_{i-1}^\red)$,
and hence the strict divisor chain
\[
1=M_0<M_1<\cdots<M_s=v
\]
together with the subsets \(J_i\) gives the desired combinatorial data. Under the recursive identification of \(\Lc(G_{i-1}^\red)\) with the formal set
\(\Lc_{i-1}\), these subsets \(J_i\) determine the required formal
combinatorial data.

It remains to show that the two constructions are inverse. This follows by
induction on \(i\). In the forward direction, the \(i\)-th step constructs
\(G_i\) from \(G_{i-1}\) using the subset \(J_i\), and the resulting element
$\hb_i=n_i\cb_i$
is equal to
$\sum_{\ub\in J_i}\ub$.
Applying the converse procedure to the resulting subgroup recovers the same
strict divisor chain and the same integers \(n_i\). At each step, both
constructions are governed by the same element \(\hb_i=n_i\cb_i\), and the
subset \(J_i\) is uniquely determined from this element by
Theorem~\ref{thm:subset-sums}. Hence the same subsets \(J_i\) are recovered.

Conversely, starting from \(G\), extracting the chain \(M_i\) and the subsets
\(J_i\), and then applying the forward construction reproduces, step by step, a
coordinate-permutation equivalent chain of subgroups, because at each stage both constructions use the same admissible element
\[
\hb_i=\sum_{\ub\in J_i}\ub=n_i\cb_i,
\]
and hence, by Lemma~\ref{lem:construction-from-h}, produce one-step extensions that are equivalent up to permutation of coordinates. Inducting on \(i\), we recover a coordinate-permutation equivalent copy of \(G\).
\end{proof}
We now turn to the counting problem. By
Theorem~\ref{thm:tower-classification}, it remains to count, for each fixed
strict divisor chain, the number of possible choices of the subsets
\(J_i\subseteq \Lc_{i-1}\).
\begin{Proposition}\label{prop:fixed-chain-count}
Fix a strict divisor chain
\[
1=M_0<M_1<\cdots<M_s=v.
\]
Then the number of possible sequences
\[
(J_1,\ldots,J_s),
\qquad
J_i\subseteq \Lc_{i-1}\ \ (i=1,\ldots,s),
\]
is \(s!\).
\end{Proposition}

\begin{proof}
Let \(A(t,r)\) be the number of possible choices after \(t\) steps such that \(|\Lc_t|=r\), and define
\[
P_t(x):=\sum_{r\ge0}A(t,r)x^r.
\]
Initially, \(P_0(x)=1\). Suppose that \(|\Lc_t|=r\). To perform the next step, we choose a subset \(J\subseteq \Lc_t\). If \(|J|=j\), then the chosen \(j\) symbols are removed and one new symbol is added, so the size changes from \(r\) to \(r-j+1\). Thus, for fixed \(r\), the contribution to the next generating polynomial is
\[
\sum_{j=0}^{r}\binom rj x^{r-j+1}
=
x(1+x)^r.
\]
Therefore
\[
P_{t+1}(x)=xP_t(1+x).
\]
By induction, for \(t\ge1\),
\[
P_t(x)=x(x+1)\cdots(x+t-1).
\]
Hence the total number of choices after \(s\) steps is
\[
P_s(1)=1\cdot2\cdots s=s!.
\]
\end{proof}

For \(s\ge1\), let
\[
c_s(D_v)
=
\#\{1=M_0<M_1<\cdots<M_s=v\}
\]
be the number of strict divisor chains of length \(s\) in the divisor lattice \(D_v\).
For a fixed strict divisor chain of length \(s\),
Proposition~\ref{prop:fixed-chain-count} shows that the associated recursive
choices contribute exactly \(s!\) classes. Summing over all strict divisor
chains in \(D_v\), we obtain the following corollary.

\begin{Corollary}\label{cor:finite-group-count}
The number of coordinate-permutation equivalence classes of finite subgroups of
type \((v,k)\) with no zero coordinate is
\[
\sum_{s\ge1}c_s(D_v)s!.
\]
In particular, it depends only on the divisor lattice \(D_v\).
\end{Corollary}

Combining Proposition~\ref{prop:type-vk-simplex} with
Theorem~\ref{thm:tower-classification} and Corollary~\ref{cor:finite-group-count}, we
immediately obtain Theorems~\ref{thm:intro-classification} and
\ref{thm:intro-main}.

\section{Examples and counting formulas}

In this section, we illustrate Theorem~\ref{thm:tower-classification} through
examples. We first recover the known cases \(v=p\), \(v=p^2\), and \(v=pq\),
then work out explicitly the next case \(v=p^3\), and finally derive closed
formulas for the number of classes in the prime-power and squarefree cases.
Here and in the examples below, for a number \(\alpha\) and a nonnegative
integer \(m\), we write \(\alpha^{[m]}\) for a block of \(m\) coordinates all
equal to \(\alpha\).

\begin{Example}[The case \(v=p\)]
Let \(p\) be a prime number. There is only one strict divisor chain, namely
\(1<p\). Hence the construction is uniquely determined, and we obtain
\[
G=
\left\langle
\left(\left(1/p\right)^{[pk]}\right)
\right\rangle.
\]
Thus there is exactly one coordinate-permutation equivalence class.
\end{Example}

\begin{Example}[The case \(v=p^2\)]
Let \(p\) be a prime number. The strict divisor chains from \(1\) to \(p^2\)
are \(1<p^2\) and \(1<p<p^2\). Hence the total number of classes is
\(1!+2!=3\).

For the chain \(1<p^2\), the group is cyclic and generated by
\[
\cb_1=\left(\left(1/{p^2}\right)^{[p^2k]}\right).
\]

For the chain \(1<p<p^2\), the first generator is
\(\cb_1=\left(\left(1/p\right)^{[pk]}\right)\). At the second step,
\(n_2=p\), and there are two choices. If \(J_2=\emptyset\), then
\[
\cb_2=
\left(
0^{[pk]},
\left(1/p\right)^{[p^2k]}
\right).
\]
If \(J_2=\{\ell_1\}\), then
\[
\cb_2=
\left(
\left(1/{p^2}\right)^{[pk]},
\left(1/p\right)^{[(p^2-1)k]}
\right).
\]
Thus there are exactly three coordinate-permutation equivalence classes.
\end{Example}

\begin{Example}[The case \(v=pq\)]
Let \(p\) and \(q\) be distinct prime numbers. The strict divisor chains from
\(1\) to \(pq\) are \(1<pq\), \(1<p<pq\), and \(1<q<pq\). Hence the total
number of classes is \(1!+2!+2!=5\).

For the chain \(1<pq\), the group is cyclic and generated by
\[
\cb_1=\left(\left(1/{pq}\right)^{[pqk]}\right).
\]

For the chain \(1<p<pq\), one starts from
\(\cb_1=\left(\left(1/p\right)^{[pk]}\right)\), and at the second step
\(n_2=q\). The two possibilities are
\[
\left(
0^{[pk]},
\left(1/q\right)^{[pqk]}
\right)
\qquad\text{and}\qquad
\left(
\left(1/{pq}\right)^{[pk]},
\left(1/q\right)^{[(pq-1)k]}
\right).
\]

For the chain \(1<q<pq\), the description is symmetric: one starts from
\(\cb_1=\left(\left(1/q\right)^{[qk]}\right)\), and at the second step
\(n_2=p\). The two possibilities are
\[
\left(
0^{[qk]},
\left(1/p\right)^{[pqk]}
\right)
\qquad\text{and}\qquad
\left(
\left(1/{pq}\right)^{[qk]},
\left(1/p\right)^{[(pq-1)k]}
\right).
\]
Thus there are exactly five coordinate-permutation equivalence classes.
\end{Example}

\begin{Example}[The case \(v=p^3\)]\label{ex:p3}
Let \(p\) be a prime number. The strict divisor chains from \(1\) to \(p^3\)
are \(1<p^3\), \(1<p<p^3\), \(1<p^2<p^3\), and \(1<p<p^2<p^3\). Hence the
total number of classes is \(1!+2!+2!+3!=11\).

We describe in detail the chain \(1<p<p^2<p^3\), which already illustrates the
whole recursive procedure. The first generator is
\(\cb_1=\left(\left(1/p\right)^{[pk]}\right)\).

If \(J_2=\emptyset\), then
\[
\cb_2=
\left(
0^{[pk]},
\left(1/p\right)^{[p^2k]}
\right),
\]
and \(\Lc_2=\{\ell_1,\ell_2\}\). At the third step, the four choices of \(J_3\) yield the following possibilities
for the third generator \(\cb_3\):
\[
\begin{array}{ll}
\emptyset:&
\displaystyle
\left(
0^{[pk]},
0^{[p^2k]},
\left(1/p\right)^{[p^3k]}
\right),\\[1.5mm]
\{\ell_1\}:&
\displaystyle
\left(
\left(1/{p^2}\right)^{[pk]},
0^{[p^2k]},
\left(1/p\right)^{[(p^3-1)k]}
\right),\\[1.5mm]
\{\ell_2\}:&
\displaystyle
\left(
0^{[pk]},
\left(1/{p^2}\right)^{[p^2k]},
\left(1/p\right)^{[(p^3-p)k]}
\right),\\[1.5mm]
\{\ell_1,\ell_2\}:&
\displaystyle
\left(
\left(1/{p^2}\right)^{[pk]},
\left(1/{p^2}\right)^{[p^2k]},
\left(1/p\right)^{[(p^3-p-1)k]}
\right).
\end{array}
\]

If \(J_2=\{\ell_1\}\), then
\[
\cb_2=
\left(
\left(1/{p^2}\right)^{[pk]},
\left(1/p\right)^{[(p^2-1)k]}
\right),
\]
and \(\Lc_2=\{\ell_2\}\). At the third step, the two choices \(J_3=\emptyset\) and \(J_3=\{\ell_2\}\)
yield the following possibilities for \(\cb_3\):
\[
\left(
0^{[pk]},
0^{[(p^2-1)k]},
\left(1/p\right)^{[p^3k]}
\right)
\qquad\text{and}\qquad
\left(
\left(1/{p^3}\right)^{[pk]},
\left(1/{p^2}\right)^{[(p^2-1)k]},
\left(1/p\right)^{[(p^3-p)k]}
\right).
\]
Thus the chain \(1<p<p^2<p^3\) gives \(6\) classes.

The remaining three chains are handled similarly. For \(1<p^3\), one obtains
the cyclic group generated by
\(\left(\left(1/{p^3}\right)^{[p^3k]}\right)\). For \(1<p<p^3\), one
starts from \(\left(\left(1/p\right)^{[pk]}\right)\), and the second
generator is either
\[
\left(
0^{[pk]},
\left(1/{p^2}\right)^{[p^3k]}
\right)
\qquad\text{or}\qquad
\left(
\left(1/{p^3}\right)^{[pk]},
\left(1/{p^2}\right)^{[(p^3-1)k]}
\right).
\]
For \(1<p^2<p^3\), one starts from
\(\left(\left(1/{p^2}\right)^{[p^2k]}\right)\), and the second generator is
either
\[
\left(
0^{[p^2k]},
\left(1/p\right)^{[p^3k]}
\right)
\qquad\text{or}\qquad
\left(
\left(1/{p^3}\right)^{[p^2k]},
\left(1/p\right)^{[(p^3-1)k]}
\right).
\]
Hence these three chains contribute \(1+2+2=5\) classes. Therefore the total
number of classes for \(v=p^3\) is \(5+6=11\).
\end{Example}
We now pass from examples to general counting formulas. In the prime-power
case, the divisor lattice is a chain, so Corollary~\ref{cor:finite-group-count}
immediately yields a closed formula.
\begin{Proposition}\label{prop:prime-power-count}
Let \(p\) be a prime number and let \(\ell\ge1\). Then
\[
N(p^\ell,k)=\sum_{s=1}^{\ell}\binom{\ell-1}{s-1}s!.
\]
\end{Proposition}

\begin{proof}
The divisor lattice of \(p^\ell\) is the chain
\(1<p<p^2<\cdots<p^\ell\). A strict divisor chain
\(1=M_0<M_1<\cdots<M_s=p^\ell\) is obtained by choosing \(s-1\) intermediate
divisors among \(p,p^2,\ldots,p^{\ell-1}\). Thus
\(c_s(D_{p^\ell})=\binom{\ell-1}{s-1}\), and the result follows from
Corollary~\ref{cor:finite-group-count}.
\end{proof}

\begin{Proposition}\label{prop:squarefree-count}
Let \(p_1,\ldots,p_\ell\) be distinct prime numbers. Then
\[
N(p_1p_2\cdots p_\ell,k)=\sum_{s=1}^{\ell}(s!)^2 S(\ell,s),
\]
where \(S(\ell,s)\) denotes the Stirling number of the second kind.
\end{Proposition}

\begin{proof}
The divisor lattice of \(p_1p_2\cdots p_\ell\) is the Boolean lattice on
\(\ell\) elements. A strict divisor chain
\(1=M_0<M_1<\cdots<M_s=p_1p_2\cdots p_\ell\) is equivalent to an ordered
partition of \(\{1,2,\ldots,\ell\}\) into \(s\) nonempty blocks, where the
\(j\)-th block records which prime factors are introduced when passing from
\(M_{j-1}\) to \(M_j\). The number of such ordered partitions is
\(s!S(\ell,s)\). Thus
\(c_s(D_{p_1p_2\cdots p_\ell})=s!S(\ell,s)\), and the result follows from
Corollary~\ref{cor:finite-group-count}.
\end{proof}

\bibliographystyle{plain}
\bibliography{bibliography}

\end{document}